\documentclass[12pt]{article}

\usepackage[margin=1in]{geometry}
\usepackage{amsmath,amsthm,amssymb}
\usepackage{enumerate}
\usepackage{xfrac}
\usepackage{accents} % use if Aball and Bball are defined below
\usepackage[colorlinks=true,hidelinks]{hyperref}
    \AtBeginDocument{}
\usepackage{tikz} % for tikz pictures
\usetikzlibrary{arrows}
\usetikzlibrary{patterns}
\usepackage{pgfplots} % for graphing functions in tikz
\usetikzlibrary{decorations.markings} % use for arrows in tikz
\usetikzlibrary{3d,calc} %3d plotting in tikz
\usepackage{pifont}
\DeclareFontFamily{OT1}{pzc}{}
  \DeclareFontShape{OT1}{pzc}{m}{it}{<-> s * [1.200] pzcmi7t}{}
  \DeclareMathAlphabet{\mathpzc}{OT1}{pzc}{m}{it}
\usepackage{fancyhdr}
\newcommand{\N}{\mathbb{N}}
\newcommand{\Z}{\mathbb{Z}}
\newcommand{\Q}{\mathbb{Q}}

\newcommand{\C}{\mathbb{C}}
\newcommand{\F}{\mathbb{F}}

\renewcommand{\P}{\mathbb{P}}

\newcommand{\G}{\mathbb{G}}

\DeclareFontFamily{U}{wncy}{}
    \DeclareFontShape{U}{wncy}{m}{n}{<->wncyr10}{}
    \DeclareSymbolFont{mcy}{U}{wncy}{m}{n}
    \DeclareMathSymbol{\Sha}{\mathord}{mcy}{"58} 

\newcommand{\Spec}{\ensuremath{\operatorname{Spec}}}

\newcommand{\orb}{\mathcal{O}}

\newcommand*\cat[1]{{\tt #1}}

\newcommand{\Sym}{\ensuremath{\operatorname{Sym}}}

\newcommand{\id}{\ensuremath{\operatorname{id}}}

\newcommand{\A}{\mathbb{A}}

\renewcommand{\div}{\ensuremath{\operatorname{div}}}

\newcommand{\W}{\mathbb{W}}

\newcommand{\eff}{\ensuremath{\operatorname{eff}}}
\newcommand{\Conf}{\ensuremath{\operatorname{Conf}}}
\newcommand{\supp}{\ensuremath{\operatorname{supp}}}
\newcommand{\Hilb}{\ensuremath{\operatorname{Hilb}}}

\renewcommand{\epsilon}{\varepsilon}
\newcommand{\legen}[2]{\left (\frac{#1}{#2}\right )}

\newtheorem{thm}{Theorem}[section]
\newtheorem{prop}[thm]{Proposition}

\newtheorem{lem}[thm]{Lemma}
\newtheorem{defn}[thm]{Definition}
\newtheorem{cor}[thm]{Corollary}

\newtheorem{rem}[thm]{Remark}
  \let\oldrem\rem
  \renewcommand{\rem}{\oldrem\normalfont}
\newtheorem{question}[thm]{Question}
\newtheorem{problem}[thm]{Problem}
\newtheorem{ex}[thm]{Example}
  \let\oldex\ex
  \renewcommand{\ex}{\oldex\normalfont}

\begin{document}

%\numberwithin{figure}{section}
%\let\stdsection\section
%\renewcommand\section{\newpage\stdsection}

%\interfootnotelinepenalty=10000

%-------------------------------------------------------------------------------------------------------------------------------------------------------------------------------------

\title{Arithmetic Functions and Geometry}
\author{Andrew Kobin}
\date{}

\maketitle

%\tableofcontents

%-------------------------------------------------------------------------------------------------------------------------------------------------------------------------------------

\rhead{\thepage}
\cfoot{}

\begin{abstract}
In this expository note, we revisit several classical arithmetic functions - namely Euler's totient function, the divisor sum functions and Dedekind’s $\psi$-function - within a unifying algebraic framework that highlights their connections to geometry. This framework builds on prior work involving zeta functions and M\"{o}bius inversion. While our main goal is to provide a clear context for similar constructions in the future, we also make an original observation regarding Dedekind’s $\psi$-function. 
\end{abstract}

\section{Introduction}

\vspace{0.1in}
In a recent article \cite{kob}, we gave a description of various zeta functions in number theory and algebraic geometry using a more abstract algebraic framework. This framework generalizes some classical notions in the theory of arithmetic functions, including the M\"{o}bius inversion principle (also see \cite{koc} for part of the story). 

In this note, we give a similar treatment to several other arithmetic functions, including the totient function (Section~\ref{sec:totient}), the divisor sum functions (Section~\ref{sec:sigmak}) and Dedekind's $\psi$-function (Section~\ref{sec:psi}). A key theme is that these functions have certain geometric origins, which we will begin discussing in Section~\ref{sec:fromgeom}. 

Many of the ideas present in this article are well-known to experts and we make no claims of originality, with one exception in Section~\ref{sec:psi}. Instead, our aim is to tell the story of these arithmetic functions through a modern, often geometric lens and motivate future study of similar functions (see Section~\ref{sec:other}).

\subsection{Generalities on arithmetic functions}

\vspace{0.1in}
Classically, an {\it arithmetic function} is any function $f : \N\rightarrow\C$, possibly with some conditions such as weak multiplicativity that we will not adopt here. To every such $f$, we can assign a Dirichlet series 
$$
F(s) = \sum_{n = 1}^{\infty} \frac{f(n)}{n^{s}}, \qquad s\in\C,
$$
which encodes the arithmetic of $f$ complex-analytically. For example, the arithmetic function $\zeta : \N\rightarrow\C$ defined by $\zeta(n) = 1$ for all $n\in\N$ corresponds to the Riemann zeta function $\zeta_{\Q}(s)$. 

One thing that makes the set of arithmetic functions interesting is that it admits the structure of a commutative $\C$-algebra by defining 
$$
(f*g)(n) = \sum_{ij = n} f(i)g(j),
$$
called the {\it Dirichlet convolution} of $f$ and $g$. More intuitively, $f*g$ is defined by ordinary multiplication of Dirichlet series: 
$$
\left (\sum_{n = 1}^{\infty} \frac{f(n)}{n^{s}}\right )\left (\sum_{n = 1}^{\infty} \frac{g(n)}{n^{s}}\right ) = \sum_{n = 1}^{\infty} \frac{(f*g)(n)}{n^{s}}. 
$$

\subsection{Abstract arithmetic functions}

\renewcommand{\P}{\mathcal{P}}

The $\C$-algebra of classical arithmetic functions is merely a special case of the more general notion of an {\it incidence algebra}, as surveyed in \cite{kob,koc}. In the most general version of the story, an algebra of ``objective arithmetic functions'' can be attached to any locally finite decomposition space (or locally finite $2$-Segal space). For concreteness, we restrict the discussion here to locally finite monoids and their vector space incidence algebras. 

Fix a field $k$. For a monoid $M$, the {\it incidence algebra} $I(M)$ is defined to be the $k$-vector space of functions $f : M\rightarrow k$ equipped with a convolution product $*$ defined by 
$$
(f*g)(x) = \sum_{ab = x} f(a)g(b)
$$
for all $x\in M$. When $M$ is a locally finite monoid\footnote{$M$ is {\it locally finite} if for each $x\in M$ there are only finitely many $a,b\in M$ with $ab = x$.}, this gives $I(M)$ the structure of an associative, unital $k$-algebra, with unit $\delta(x) = \delta_{x,1}$, where $1$ is the identity element of $M$. 

Each locally finite monoid $M$ admits a {\it zeta function} $\zeta\in I(M)$ defined by $\zeta(x) = 1$ for all $x\in M$. M\"{o}bius inversion asserts that $\zeta$ is invertible in $I(M)$, with inverse given by 
$$
\mu(x) = \begin{cases}
    1, &\text{if } x = 1\\
    -\sum_{ab = x} \mu(a), &\text{if } x\not = 1. 
\end{cases}
$$
The usual application of M\"{o}bius inversion is to invert a relation between two functions $f,g\in I(M)$, e.g.~convolving both sides of $f = g*\zeta$ with $\mu$ yields $f*\mu = g$, or: 
\begin{equation}\label{eq:mobius}
f(x) = \sum_{ab = x} g(a) \quad\text{implies}\quad g(x) = \sum_{ab = x} f(a)\mu(b). 
\end{equation}
The goal of this short note is to highlight several pairs of arithmetic functions $(f,g)$ satisfying formula (\ref{eq:mobius}) and extend them to a more general arithmetic geometric context.

\subsection{Arithmetic functions coming from geometry}
\label{sec:fromgeom}

Many arithmetic functions ``come from geometry'', a phrase which we will not give precise meaning to. Instead, we will explore the geometric flavor of several examples, even when this connection is not typically mentioned in the literature. 

To begin with, the functions $\zeta$ and $\mu$ that form the basis for M\"{o}bius inversion have well-known geometric interpretations. The Riemann zeta function $\zeta_{\Q}(s)$ can be viewed as the zeta function $\zeta_{\Spec\Z}(s)$ attached to the scheme $\Spec\Z$ and more generally, Dedekind zeta functions of number fields and Hasse--Weil zeta functions of varieties can be interpreted as zeta functions for schemes over $\Spec\Z$. This often provides a geometric explanation for relations between these different functions; cf.~\cite{kob}. 

In fact, these zeta functions are mere shadows\footnote{In technical terms, they are decategorifications of generating functions whose coefficients lie in some Grothendieck ring.} of an abstract zeta function, such as Kapranov's motivic zeta function 
$$
Z_{mot}(X,t) = 1 + \sum_{n = 1}^{\infty} [\Sym^{n}X]t^{n},
$$
which lies in the power series ring $1 + tK_{0}(\cat{Sch}_{X})[[t]]$, where $K_{0}(\cat{Sch}_{X})$ is the Grothendieck ring of the category of $X$-schemes. Here, $\Sym^{n}X$ denotes the $n$th symmetric power of $X$, parametrizing unordered $n$-tuples of points in $X$. Choosing $X$ appropriately and applying a motivic measure to the coefficients of $Z_{mot}(X,t)$ recovers the more familiar zeta functions in the previous paragraph. 

Like $\zeta$, the classical M\"{o}bius function $\mu : \N\rightarrow\C$ is an element of the incidence algebra $I(\N^{\times})$, where $\N^{\times}$ is the multiplicative monoid of natural numbers. It too is a shadow of a motivic generating function, namely 
$$
M_{mot}(X,t) = 1 + \sum_{\mu} (-1)^{||\mu||}[\Sym^{\mu}X]t^{|\mu|} = 1 + \sum_{n = 1}^{\infty} \sum_{\mu\vdash n} (-1)^{||\mu||}[\Sym^{\mu}X]t^{n}
$$
where the sum is over all partitions $\mu = (\mu_{1},\ldots,\mu_{k})$ of natural numbers $n\in\N$, with $|\mu| := \sum_{i = 1}^{k} \mu_{i} = n$, $||\mu||$ equal to the number of distinct elements of $\mu$ and $\Sym^{\mu}X := \coprod_{i = 1}^{k} \Sym^{\mu_{i}}X$. This can be rewritten \cite[Prop.~3.7]{vw} 
$$
M_{mot}(X,t) = 1 + \sum_{\mu} (-1)^{||\mu||}[\Conf^{\mu}X]t^{|\mu|},
$$
where $\Conf^{\mu}X\subseteq\Sym^{\mu}X$ is a {\it configuration space}, i.e.~the open subscheme parametrizing distinct tuples of points. %In turn, this implies 
%$$
%\log M_{mot}(X,t) = -\sum_{n = 1}^{\infty} \frac{[\Conf^{n}X]}{n}t^{n},
%$$
%where $\Conf^{n}X\subseteq\Sym^{n}X$ is defined as above. 
For a lengthier discussion of generating functions involving configuration spaces, see \cite[Sec.~5]{vw}. 

\begin{rem}
As an amusing sidenote, the other famous mathematical object named for August Ferdinand M\"{o}bius, the {\it M\"{o}bius strip}, is also a configuration space: $\Conf^{2}(S^{1})$. 
\end{rem}

Another example worth mentioning here is the partition function $p(n)$, which counts the number of ways of writing $n$ as a sum of positive integers. On the geometric side, the values of $p(n)$ are encoded in the Hilbert schemes $\Hilb^{n}\A^{2}$, parametrizing collections of $n$ unordered points\footnote{To be more accurate, while $\Sym^{n}\A^{2}$ parametrizes $n$ unordered points in the plane, $\Hilb^{n}\A^{2}$ parametrizes such points \emph{together with an infinitesimal structure when points collide}. In fancier terms, the natural map $\Hilb^{n}\A^{2}\rightarrow\Sym^{n}\A^{2}$ is a desingularization.} in $\A^{2}$. Specifically, $p(n) = \chi(\Hilb^{n}\A^{2})$, where $\chi$ denotes the Euler characteristic. We can thus view the partition generating function $\sum_{n = 0}^{\infty} p(n)t^{n}$, originally studied by Euler, as a shadow of the motivic generating function 
$$
P_{mot}(\A^{2},t) = 1 + \sum_{n = 1}^{\infty} [\Hilb^{2}\A^{2}]t^{n}. 
$$
For a version of this relation when $\A^{2}$ is replaced by a smooth surface $X$, see \cite{got}.

\subsection{Overview}

\renewcommand{\P}{\mathbb{P}}

For the rest of this article, we will highlight more arithmetic functions and discuss their possible geometric origins. We summarize things in the following table. 
\begin{center}
\begin{tabular}{|c|c|c|c|}
    \hline
    Function & Symbol & Formula & Geometric interpretation\\
    \hline
    zeta & $\zeta$ & $\zeta(n) = 1$ & symmetric powers\\
    (Sec.~\ref{sec:fromgeom}) & & & \\
    \hline
    $\begin{matrix} \text{M\"{o}bius} \\ \text{(Sec.~\ref{sec:fromgeom})}\end{matrix}$ & $\mu$ & $\mu(n) = \begin{cases}
        1, &\text{if } n = 1\\
        -\sum_{ab = n} \mu(a), &\text{if } n\not = 1
    \end{cases}$ & configuration spaces\\
    \hline
    partition & $p$ & $p(n) = \#\{\lambda\vdash n\}$ & Hilbert schemes\\
    (Sec.~\ref{sec:fromgeom}) & & & \\
    \hline
    totient & $\phi$ & $\phi(n) = \#\{1\leq r\leq n\mid (r,n) = 1\}$ & multiplicative group $\G_{m}$\\
    (Sec.~\ref{sec:totient}) & & & \\
    \hline
    divisor sum & $\sigma_{1}$ & $\sigma_{1}(n) = \sum_{d\mid n} d$ & projective line $\P^{1}$\\
    (Sec.~\ref{sec:sigmak}) & $\sigma_{m}$ & $\sigma_{m}(n) = \sum_{d\mid n} d^{m}$ & reducible scheme $\A^{m}\amalg *$\\
    \hline
    psi & $\psi$ & $\psi(n) = n\prod_{p\mid n} (1 + p^{-1})$ & Witt vector construction\\
    (Sec.~\ref{sec:psi}) & & & \\
    \hline
    lambda & $\lambda$ & $\lambda(n) = (-1)^{\Omega(n)}$ & ??\\
    (Sec.~\ref{sec:lambda}) & & & \\
    \hline
    sum of squares & $r_{2}$ & $r_{2}(n) = \#\{(x,y) \mid n = x^{2} + y^{2}\}$ & ??\\
    (Sec.~\ref{sec:sumofsquares}) & & & \\
    \hline
\end{tabular}
\end{center}

%-------------------------------------------------------------------------------------------------------------------------------------------------------------------------------------

\section{Totient Functions}
\label{sec:totient}

Euler's {\it totient function} $\phi : \N\rightarrow\Z$ was originally defined as: 
$$
\phi(n) = \#\{1\leq r\leq n\mid (r,n) = 1\}. 
$$
Alternatively, $\phi$ can be defined recursively by 
\begin{equation}\label{eq:totientrecur}
n = \sum_{d\mid n} \phi(d), 
\end{equation}
or $\id = \phi*\zeta$ as arithmetic functions, where $\zeta(n) = 1$ is the zeta function in the incidence algebra $I(\N^{\times})$. Then it is immediate from M\"{o}bius inversion that: 
\begin{equation}\label{eq:classicalmobius}
\phi(n) = \sum_{d\mid n} d\mu\left (\frac{n}{d}\right )
\end{equation}
where $\mu$ is the classical M\"{o}bius function. That is, $\phi = \id*\mu$ as arithmetic functions. 

Another way of packaging this information uses the generating function of $\phi$, revealing a connection to the Riemann zeta function. 

\begin{prop}
\label{prop:phizetaformula}
As formal Dirichlet series, 
$$
\sum_{n = 1}^{\infty} \frac{\phi(n)}{n^{s}} = \frac{\zeta_{\Q}(s - 1)}{\zeta_{\Q}(s)}. 
$$
\end{prop}

\begin{proof}
By Dirichlet convolution, 
$$
\sum_{n = 1}^{\infty} \frac{\phi(n)}{n^{s}} = \left (\sum_{n = 1}^{\infty} \frac{n}{n^{s}}\right )\left (\sum_{n = 1}^{\infty} \frac{\mu(n)}{n^{s}}\right ) = \zeta_{\Q}(s - 1)\mu_{\Q}(s) = \zeta_{\Q}(s - 1)\zeta_{\Q}(s)^{-1}. 
$$
\end{proof}

The classical totient function has several alternative descriptions which are ripe for generalization. For a ring $A$, let $A^{\times}$ denote the group of units of $A$ under multiplication. Then $\phi(n) = (\Z/n\Z)^{\times}$. Even better, Sun Tzu's Theorem\footnote{This is sometimes called the Chinese Remainder Theorem. Sun Tzu (also written Sun Zi or Sunzi) is not to be confused with the author of the same name.} says that if $n = p_{1}^{k_{1}}\cdots p_{r}^{k_{r}}$ for distinct primes $p_{i}$ and $k_{i}\geq 1$, 
$$
(\Z/n\Z)^{\times} \cong (\Z/p_{1}^{k_{1}}\Z)^{\times}\times\cdots(\Z/p_{r}^{k_{r}}\Z)^{\times}. 
$$
In particular, $\phi$ is a (weakly) multiplicative arithmetic function. One also has that for any prime $p$ and $k\geq 1$, $\phi(p^{k}) = p^{k - 1}(p - 1) = p^{k}\left (1 - \frac{1}{p}\right )$, implying the following product formula: 
\begin{equation}\label{eq:totientprod}
\phi(n) = n\prod_{p\mid n} \left (1 - \frac{1}{p}\right )
\end{equation}
where the product runs over all prime factors of $n$.

%-------------------------------------------------------------------------------------------------------------------------------------------------------------------------------------

\subsection{Totient function of a number field}
\label{sec:dedekindtotient}

Let $K/\Q$ be a number field with ring of integers $\orb_{K}$ and let $I_{K}^{\times}$ denote the monoid of nonzero ideals of $\orb_{K}$ under multiplication. The ideal norm $N_{K/\Q}$ induces a morphism of monoids $N : I_{K}^{+}\rightarrow\N$ which can also be viewed as an arithmetic function on $I_{K}^{\times}$ by composing with a fixed map $\N\rightarrow k$, usually the homomorphism sending $1\mapsto 1$. We will abuse notation and write $N$ for both versions. 

Recall that for an ideal $\frak{a}\in I_{K}^{+}$, $N(\frak{a}) = \#(\orb_{K}/\frak{a})$. If we think of $N(\frak{a})$ as a generalization of a number $n = \#(\Z/n\Z)$, then a reasonable generalization of $\phi(n) = \#(\Z/n\Z)^{\times}$ is the following. 

\begin{defn}
The {\bf totient function of a number field} $K/\Q$ is the function $\phi_{K} : I_{K}^{+}\rightarrow\Z$ defined by 
$$
\phi_{K}(\frak{a}) = \#(\orb_{K}/\frak{a})^{\times}. 
$$
\end{defn}

Viewing $\phi_{K}$ as an arithmetic function in the incidence algebra $I(I_{K}^{\times})$ allows us to determine many of its properties and relations to other arithmetic functions. 

Recall that two ideals $\frak{a}$ and $\frak{b}$ of a ring $A$ are {\it relatively prime} if $\frak{a} + \frak{b} = A$. Applying Sun Tzu's Theorem to the ring $\orb_{K}/\frak{ab}$ shows that $\phi_{K}$ is multiplicative, i.e.~for any relatively prime ideals $\frak{a},\frak{b}\subset\orb_{K}$, $\phi_{K}(\frak{ab}) = \phi_{K}(\frak{a})\phi_{K}(\frak{b})$. It follows immediately that for any prime ideal $\frak{p}$ in $\orb_{K}$, $\phi_{K}(\frak{p}) = N(\frak{p}) - 1$ and more generally, 
\begin{equation}
\phi_{K}(\frak{p}^{k}) = N(\frak{p})^{k - 1}(N(\frak{p}) - 1) = N(\frak{p})^{k}\left (1 - \frac{1}{N(\frak{p})}\right ). 
\end{equation}
As a consequence, for any ideal $\frak{a}\subset\orb_{K}$, 
\begin{equation}\label{eq:Dedekindtotientprod}
\phi_{K}(\frak{a}) = N(\frak{a})\prod_{\frak{p}\mid\frak{a}} \left (1 - \frac{1}{N(\frak{p})}\right ). 
\end{equation}

On the other hand, $\phi_{K}$ can be described by a recursive formula similar to formula (\ref{eq:totientrecur}). 

\begin{prop}
For any ideal $\frak{a}\in I_{K}^{+}$, 
\begin{equation}
N(\frak{a}) = \sum_{\frak{d}\mid\frak{a}} \phi_{K}(\frak{d})
\end{equation}
where the sum runs over all ideals $\frak{d}$ dividing $\frak{a}$. In other words, $N = \phi_{K}*\zeta$ where $\zeta(\frak{a}) = 1$ is the zeta function for $I_{K}^{+}$. 
\end{prop}

\begin{proof}
By formula (\ref{eq:Dedekindtotientprod}), 
$$
\phi_{K}(\frak{a}) = N(\frak{a})\prod_{\frak{p}\mid\frak{a}} \left (1 - \frac{1}{N(\frak{p})}\right ) = N(\frak{a})\sum_{\frak{d}\mid\frak{a}} \frac{\mu(\frak{d})}{N(\frak{d})} = \sum_{\frak{d}\mid\frak{a}} \mu(\frak{d})N(\frak{ad}^{-1}). 
$$
This can be rewritten 
\begin{equation}
\phi_{K}(\frak{a}) = \sum_{\frak{d}\mid\frak{a}} N(\frak{d})\mu(\frak{ad}^{-1})
\end{equation}
i.e.~$\phi_{K} = N*\mu$. Then M\"{o}bius inversion gives the desired formula. 
\end{proof}

\begin{rem}
We could have proved formula (\ref{eq:totientrecur}) for the classical totient function this way too, i.e.~by first proving formula (\ref{eq:classicalmobius}) directly using the above argument and then applying M\"{o}bius inversion. 
\end{rem}

The ideal norm, viewed as a monoid map $N : I_{K}^{\times}\rightarrow\N^{\times}$, allows one to build Dirichlet series out of arithmetic functions in $I(I_{K}^{\times})$, as explained in \cite[Rmk.~3.31]{kob}. For $f\in I(I_{K}^{\times})$, its associated Dirichlet series is the Dirichlet series attached to the {\it pushforward} $N_{*}f$, defined by 
$$
(N_{*}f)(n) = \sum_{N(\frak{a}) = n} f(\frak{a}). 
$$
The explicit Dirichlet series can be written 
$$
F(s) = \sum_{n = 1}^{\infty} \frac{(N_{*}f)(n)}{n^{s}} = \sum_{\frak{a}\in I_{K}^{\times}} \frac{f(\frak{a})}{N(\frak{a})^{s}}. 
$$

As in Proposition~\ref{prop:phizetaformula}, the Dirichlet series for $\phi_{K}$ can be expressed in terms of the Dedekind zeta function for $K/\Q$. 

\begin{prop}
\label{prop:Dedekindphizetaformula}
As formal Dirichlet series, 
\begin{equation}
\sum_{\frak{a}\in I_{K}^{+}} \frac{\phi_{K}(\frak{a})}{N(\frak{a})^{s}} = \frac{\zeta_{K}(s - 1)}{\zeta_{K}(s)}. 
\end{equation}
\end{prop}

\begin{proof}
Since $\phi_{K} = N*\mu$, we get 
$$
\sum_{\frak{a}\in I_{K}^{+}} \frac{\phi_{K}(\frak{a})}{N(\frak{a})^{s}} = \left (\sum_{\frak{a}\in I_{K}^{+}} \frac{N(\frak{a})}{N(\frak{a})^{s}}\right )\left (\sum_{\frak{a}\in I_{K}^{+}} \frac{\mu(\frak{a})}{N(\frak{a})^{s}}\right ) = \zeta_{K}(s - 1)\mu_{K}(s) = \zeta_{K}(s - 1)\zeta_{K}(s)^{-1}. 
$$
\end{proof}

More generally, a totient function can be defined on the monoid of $0$-dimensional ideals in an {\it abstract number ring}, in the sense of \cite{cla}. For example, the coordinate ring of a curve $C/\F_{q}$ admits such a totient function. Viewing the corresponding incidence algebra more abstractly allows us to generalize the definition to arbitrary varieties over a finite field in the next section.

%-------------------------------------------------------------------------------------------------------------------------------------------------------------------------------------

\subsection{Totient function of a variety}
\label{sec:totientvariety}

Let $X$ be a variety over $k = \F_{q}$. The definition of the totient function in this section is a generalization of the one given in \cite{ar} for $X$ a curve. 

For a closed point $x\in |X|$, let $\orb_{X,x}$ be the local ring at $x$, $\frak{m}_{x}$ the maximal ideal corresponding to $x$ and $k(x) = \orb_{X,x}/\frak{m}_{x}$ the residue field at $x$, which is a finite extension of $k$, namely $k(x) = \F_{q^{\deg(x)}}$. 

Let $Z_{0}^{\eff}(X)$ be the monoid of effective $0$-cycles on $X$ under addition. For an effective $0$-cycle $\alpha = \sum_{i = 1}^{r} a_{i}x_{i}$, where $x_{i}$ are distinct closed points of $X$ and $a_{i}\geq 1$, define the semilocal ring associated to $\alpha$ by 
$$
\orb_{X,\alpha} := \bigcap_{i = 1}^{r} \orb_{X,x_{i}}
$$
and the distinguished ideal of $\orb_{X,\alpha}$ by 
$$
I_{\alpha} = \bigcap_{i = 1}^{r} \frak{m}_{x_{i}}^{a_{i}}. 
$$
Then $\orb_{X,\alpha}/I_{\alpha}$ is a finite ring of order $q^{\deg(\alpha)}$. In this way, we can regard the function $\pi : Z_{0}^{\eff}(X)\rightarrow\N,\alpha\mapsto q^{\deg(\alpha)}$ as an analogue of the norm $N_{K/\Q}$ of a number field $K$. We will give a careful description of this in a moment. 

\begin{defn}
The {\bf totient function of a variety} $X/\F_{q}$ is the function $\phi_{X} : Z_{0}^{\eff}(X)\rightarrow\Z$ defined by 
$$
\phi_{X}(\alpha) = \#(\orb_{X,\alpha}/I_{\alpha})^{\times}. 
$$
\end{defn}

As in Section~\ref{sec:dedekindtotient}, we may extract certain algebraic properties of $\phi_{X}$ by viewing it as an arithmetic function in the incidence algebra $I(Z_{0}^{\eff}(X))$. 

We say two $0$-cycles $\alpha,\beta\in Z_{0}^{\eff}(X)$ are {\it relatively prime} if $\supp(\alpha)\cap\supp(\beta) = \varnothing$. Then Sun Tzu's Theorem\footnote{In \cite[Lem.~7]{ar}, the authors give a careful proof of Sun Tzu's Theorem in the dimension $1$ case but the proof carries over verbatim to the finite rings $\orb_{X,\alpha}/I_{\alpha}$.} implies that if $\alpha$ and $\beta$ are relatively prime, $\phi_{X}(\alpha + \beta) = \phi_{X}(\alpha)\phi_{K}(\beta)$. It follows that for any closed point $x\in |X|$, $\phi_{X}(x) = q^{\deg(x)} - 1$. More generally, 
\begin{equation}
\phi_{X}(ax) = q^{\deg(ax)}(1 - q^{-\deg(x)}).     
\end{equation}
From this, we get: 
\begin{equation}\label{eq:HWtotientprod}
\phi_{X}(\alpha) = q^{\deg(\alpha)}\prod_{i = 1}^{r} (1 - q^{-\deg(x_{i})})
\end{equation}
for any effective $0$-cycle $\alpha = \sum_{i = 1}^{r} a_{i}x_{i}$ on $X$. 

As before, there is also a recursive formula for $\phi_{X}$. 

\begin{prop}
For any effective $0$-cycle $\alpha$ on $X$, 
\begin{equation}
q^{\deg(\alpha)} = \sum_{\beta\leq\alpha} \phi_{X}(\beta). 
\end{equation}
In other words, $\pi = \phi_{X}*\zeta$, where $\zeta(\alpha) = 1$ is the zeta function for $Z_{0}^{\eff}(X)$ and $\pi = q^{\deg(-)}$. 
\end{prop}

\begin{proof}
Write $\alpha = \sum_{i = 1}^{r} a_{i}x_{i}$. By formula (\ref{eq:HWtotientprod}), 
$$
\phi_{X}(\alpha) = q^{\deg(\alpha)}\prod_{i = 1}^{r} (1 - q^{-\deg(x_{i})}) = q^{\deg(\alpha)}\sum_{\beta\leq\alpha} \mu(\beta)q^{-\deg(\beta)} = \sum_{\beta\leq\alpha} \mu(\beta)q^{\deg(\alpha - \beta)}
$$
where $\mu$ is the M\"{o}bius function for $Z_{0}^{\eff}(X)$. This can be rewritten 
\begin{equation}
\phi_{X}(\alpha) = \sum_{\beta\leq\alpha} \mu(\beta)q^{\deg(\alpha - \beta)} = \sum_{\beta\leq\alpha} q^{\deg(\beta)}\mu(\alpha - \beta). 
\end{equation}
Or in other words, $\phi_{X} = \pi*\mu$. Then M\"{o}bius inversion for $Z_{0}^{\eff}(X)$ implies that $\pi = \phi_{X}*\zeta$ as claimed. 
\end{proof}

The analogue of the field norm $N_{K/\Q}$ in this context is the terminal map $t : X\rightarrow\Spec k$, which induces a map $\pi = t_{*} : Z_{0}^{\eff}(X)\rightarrow Z_{0}^{\eff}(\Spec k)$. In turn, this induces a pushforward map $\pi_{*} : I(Z_{0}^{\eff}(X))\rightarrow I(Z_{0}^{\eff}(\Spec k))$. Identifying $Z_{0}^{\eff}(\Spec k) \cong \N_{0}$ as additive monoids, $\pi_{*}$ associates to any $f\in I(Z_{0}^{\eff}(X))$ a power series in $I(Z_{0}^{\eff}(\Spec\F_{q})) \cong k[[t]]$ via the assignment 
$$
\pi_{*}f \longleftrightarrow \sum_{n = 0}^{\infty} (\pi_{*}f)(n)t^{n} = \sum_{\alpha\in Z_{0}^{\eff}(X)} f(\alpha)t^{\deg(\alpha)}. 
$$
From this, we recover a formula relating $\phi_{X}$ to the the Hasse--Weil zeta function of $X$. 

\begin{prop}
\label{prop:HWphizetaformula}
For any variety $X$ over a finite field, 
$$
\sum_{\alpha\in Z_{0}^{\eff}(X)} \phi_{X}(\alpha)t^{\deg(\alpha)} = \frac{Z(X\times\A^{1},t)}{Z(X,t)}. 
$$
\end{prop}

\begin{proof}
Translating the identity $\phi_{X} = \pi*\mu$ to power series produces 
\begin{align*}
  \sum_{\alpha\in Z_{0}^{\eff}(X)} \phi_{X}(\alpha)t^{\deg(\alpha)} &= \left (\sum_{\alpha\in Z_{0}^{\eff}(X)} q^{\deg(\alpha)}t^{\deg(\alpha)}\right )\left (\sum_{\alpha\in Z_{0}^{\eff}(X)} \mu(\alpha)t^{\deg(\alpha)}\right )\\
    &= \left (\sum_{\alpha\in Z_{0}^{\eff}(X)} (qt)^{\deg(\alpha)}\right )Z(X,t)^{-1} = Z(X,qt)Z(X,t)^{-1}.
\end{align*}
It is well-known that $Z(X,qt) = Z(X\times\A^{1},t)$, so the claimed formula holds. 
\end{proof}

%-------------------------------------------------------------------------------------------------------------------------------------------------------------------------------------

\subsection{Totients and arithmetic schemes}

Let $X$ be an arithmetic scheme, i.e.~a scheme of finite type over $\Spec\Z$. For each prime $p$, the $\F_{p}$-scheme $X_{p}$ has a totient function $\phi_{X_{p}} : Z_{0}^{\eff}(X_{p})\rightarrow\Z$ defined by $\phi_{X_{p}}(\alpha) = \#(\orb_{X_{p},\alpha}/I_{\alpha})^{\times}$ for any $\alpha\in Z_{0}^{\eff}(X_{p})$. Denote its generating function by 
$$
\Phi(X_{p},t) = \sum_{\alpha\in Z_{0}^{\eff}(X_{p})} \phi_{X_{p}}(\alpha)t^{\deg(\alpha)}. 
$$
These can be packaged together into a Dirichlet series, 
$$
\Phi_{X}(s) = \prod_{p\text{ prime}} \Phi(X_{p},p^{-s}). 
$$
From formula (\ref{eq:HWtotientprod}), it follows that 
\begin{equation}
\Phi_{X}(s) = \prod_{x\in |X|} \left (1 - \frac{N(x) - 1}{N(x)^{s}}\right )^{-1}. 
\end{equation}
Additionally, Proposition~\ref{prop:HWphizetaformula} implies that 
\begin{equation}\label{eq:arithphizeta}
\Phi_{X}(s) = \frac{\zeta_{X}(s - 1)}{\zeta_{X}(s)}. 
\end{equation}

\begin{rem}
\label{rem:totientmotivic}
Thinking motivically, we can see formula (\ref{eq:arithphizeta}) and its various special cases (Propositions~\ref{prop:phizetaformula}, \ref{prop:Dedekindphizetaformula} and~\ref{prop:HWphizetaformula}) as shadows of the identity $[\A^{1}] = [\G_{m}] + [\text{point}]$ in the Grothendieck ring of $\Z$-schemes (resp.~$\orb_{K}$-schemes, $\F_{q}$-varieties). 
\end{rem}

%-------------------------------------------------------------------------------------------------------------------------------------------------------------------------------------

\subsection{Application: Euler's theorem}

One classical application of Euler's totient function is the formula
\begin{equation}\label{eq:euler}
a^{\phi(n)} \equiv 1\pmod{n},
\end{equation}
which holds for every $n\geq 2$ and $a$ relatively prime to $n$. This admits the following generalizations, which are routine to verify. 

\begin{prop}
\label{prop:eulernumberfield}
Let $K/\Q$ be a number field with ring of integers $\orb_{K}$. For any $\alpha\in\orb_{K}$ and ideal $\frak{b}$ such that $\alpha\orb_{K}$ and $\frak{b}$ are relatively prime, 
$$
\alpha^{\phi_{K}(\frak{b})} \equiv 1\pmod{\frak{b}}. 
$$
\end{prop}

\begin{prop}
\label{prop:eulervariety}
Let $X/\F_{q}$ be a variety, $U\subseteq X$ an open subset and $f\in\orb_{X}(U)$. Then for any effective $0$-cycle $\beta\in Z_{0}^{\eff}(X)$ with support outside $V(f) = \{x\in U\mid f(x) = 0\}$, 
$$
f^{\phi_{X}(\beta)} \equiv 1\pmod{I_{\beta}}
$$
in the semilocal ring $\orb_{X,\beta}$. 
\end{prop}

As a special case of Proposition~\ref{prop:eulervariety}, we recover the following version of Euler's theorem for polynomials over a finite field, originally due to Wardlaw \cite{war}. 

\begin{cor}
Let $k$ be a finite field, $g\in k[t]$ and $D = \div_{0}(g)$ the effective divisor of zeroes of $g$. Then 
\begin{enumerate}[\quad (1)]
    \item $\phi_{\P^{1}}(D) = \#\{f\in k[t]\mid \deg(f) < \deg(g) \text{ and } \gcd(f,g) = 1\}$. 
    \item For any such $f\in k[t]$, 
    $$
    f^{\phi_{\P^{1}}(D)} \equiv 1\pmod{g}
    $$
    holds in $k[t]$. 
\end{enumerate}
\end{cor}

\begin{proof}
Both statements follow from setting $X = \P^{1}$, $U = \P^{1}\smallsetminus\{\infty\} \cong \A^{1}$, identifying the rings $\orb_{\P^{1},D} \cong k[t]/(g)$ and taking cardinalities of multiplicative groups of units. 
\end{proof}

%-------------------------------------------------------------------------------------------------------------------------------------------------------------------------------------

\section{Divisor Sum Functions}
\label{sec:sigmak}

The divisor sum functions 
$$
\sigma_{m}(n) = \sum_{d\mid n} d^{m}
$$
are another source of interesting Dirichlet series, namely 
\begin{equation}\label{eq:divisorsumDirichletseries}
\sum_{n = 1}^{\infty} \frac{\sigma_{m}(n)}{n^{s}} = \zeta_{\Q}(s)\zeta_{\Q}(s - m). 
\end{equation}
In particular, the Dirichlet series for $\sigma_{1}$ is $\zeta_{\Q}(s)\zeta_{\Q}(s - 1) = \zeta_{\P^{1}}(s)$, the zeta function for the arithmetic scheme $\P^{1}/\Spec\Z$. This agrees with the formula $\sigma_{1}(n) = \#\P^{1}(\Z/n\Z)$, which we will generalize in a moment. 

Note that $\sigma_{1}$ is weakly multiplicative and for prime powers, it takes on the values $\sigma_{1}(p^{k}) = 1 + p + \ldots + p^{k} = \frac{p^{k + 1} - 1}{p - 1}$. Together these imply that for any $n\geq 1$, 
\begin{equation}
\sigma_{1}(n) = \prod_{p^{k}\mid\mid n} \frac{p^{k + 1} - 1}{p - 1},
\end{equation}
where $p^{k}\mid\mid n$ denotes that $p^{k}\mid n$ but $p^{k + 1}\nmid n$. 

By definition, $\sigma_{1} = \id*\zeta$ as arithmetic functions. Then M\"{o}bius inversion implies $\id = \sigma_{1}*\mu$, or equivalently, 
\begin{equation}
n = \sum_{d\mid n} \sigma_{1}(d)\mu\left (\frac{n}{d}\right ). 
\end{equation}

%-------------------------------------------------------------------------------------------------------------------------------------------------------------------------------------

\subsection{Divisor sum of a number field}

\begin{defn}
The {\bf first divisor sum of a number field} $K/\Q$ is the arithmetic function $\sigma_{1,K} : I_{K}^{\times}\rightarrow\Z$ defined by 
$$
\sigma_{1,K}(\frak{a}) = \#\P^{1}(\orb_{K}/\frak{a}). 
$$
\end{defn}

For any relatively prime ideals $\frak{a},\frak{b}\subset\orb_{K}$, we have\footnote{More generally, if $T$ is a scheme and $f_{T}$ is the arithmetic function $f_{T}(n) = \#T(\Z/n\Z)$, then $f_{T}$ is weakly multiplicative by Sun Tzu's Theorem and the fact that the functor of points of $T$ sends coproducts to products.} $\sigma_{1,K}(\frak{ab}) = \sigma_{1,K}(\frak{a})\sigma_{1,K}(\frak{b})$. From the definition, it follows that for any prime ideal $\frak{p}\in I_{K}^{+}$, $\sigma_{1,K}(\frak{p}) = N(\frak{p}) + 1$, and for any $i\geq 1$, 
\begin{equation}
\sigma_{1,K}(\frak{p}^{k}) = 1 + N(\frak{p}) + \ldots + N(\frak{p})^{k} = \frac{N(\frak{p})^{k + 1} - 1}{N(\frak{p}) - 1}. 
\end{equation}
That is, for any ideal $\frak{a}\in I_{K}^{+}$, 
\begin{equation}
\sigma_{1,K}(\frak{a}) = \prod_{\frak{p}^{k}\mid\mid\frak{a}} \frac{N(\frak{p})^{k + 1} - 1}{N(\frak{p}) - 1}. 
\end{equation}
In terms of arithmetic functions in $I(I_{K}^{\times})$, $\sigma_{1,K} = N*\zeta$, and as a result, we obtain the following. 
\begin{prop}
As formal Dirichlet series, 
$$
\sum_{\frak{a}\in I_{K}^{+}} \frac{\sigma_{1,K}(\frak{a})}{N(\frak{a})^{s}} = \zeta_{K}(s)\zeta_{K}(s - 1). 
$$
\end{prop}

%-------------------------------------------------------------------------------------------------------------------------------------------------------------------------------------

\subsection{Divisor sum of a variety}

Next, let $X$ be a variety over $k = \F_{q}$, $x\in |X|$ a closed point and let $\orb_{X,x}$, $\frak{m}_{x}$, $k(x)$, $Z_{0}^{\eff}(X)$, $\orb_{X,\alpha}$ and $I_{\alpha}$ be as before. 

\begin{defn}
The {\bf first divisor sum of a variety} $X/\F_{q}$ is the function $\sigma_{1,X} : Z_{0}^{\eff}(X)\rightarrow\Z$ defined by 
$$
\sigma_{1,X}(\alpha) = \#\P^{1}(\orb_{X,\alpha}/I_{\alpha}). 
$$
\end{defn}

One can check that for any relatively prime $0$-cycles $\alpha,\beta\in Z_{0}^{\eff}(X)$, $\sigma_{1,X}(\alpha + \beta) = \sigma_{1,X}(\alpha)\sigma_{1,X}(\beta)$, and for a closed point $x$, $\sigma_{1,X}(x) = q^{\deg(x)} + 1$. Together, these imply that for any effective $0$-cycle $\alpha = \sum_{i = 1}^{r} a_{i}x_{i}$ on $X$, 
\begin{equation}
\sigma_{1,X}(\alpha) = \prod_{i = 1}^{r} \frac{q^{\deg(x_{i}) + 1} - 1}{q - 1}. 
\end{equation}

As in Section~\ref{sec:totientvariety}, let $\pi : Z_{0}^{\eff}(X)\rightarrow Z_{0}^{\eff}(\Spec\F_{q})$ be the monoid map induced by $t : X\rightarrow\Spec\F_{q}$. Viewing this as an arithmetic function on $Z_{0}^{\eff}(X)$, we see from the definition that $\sigma_{1,X} = \pi*\zeta$, which implies the following. 
\begin{prop}
\label{prop:sigma1variety}
For any variety $X/\F_{q}$, 
$$
\sum_{\alpha\in Z_{0}^{\eff}(X)} \sigma_{1,K}(\alpha)t^{\deg(\alpha)} = Z(X,t)Z(X\times\A^{1},t). 
$$
\end{prop}

%-------------------------------------------------------------------------------------------------------------------------------------------------------------------------------------

\subsection{Global divisor sums}

Finally, since $\P^{1}$ is an arithmetic scheme with good reduction everywhere (given by $\P_{\F_{p}}^{1} = \P^{1}\times_{\Z}\Spec\F_{p}$), it makes sense to define a global divisor sum. 

\begin{defn}
For an arithmetic scheme $X$, the {\bf global first divisor sum} of $X$ is the formal power series 
$$
S_{1,X}(s) = \prod_{p\text{ prime}} S_{1}(X_{p},p^{-s})
$$
where for each prime $p$, 
$$
S_{1}(X_{p},t) = \sum_{\alpha\in Z_{0}^{\eff}(X_{p})} \sigma_{1,X_{p}}(\alpha)t^{\deg(\alpha)}. 
$$
\end{defn}

\begin{cor}
For any arithmetic scheme $X$, we have 
$$
S_{1,X}(s) = \zeta_{X}(s)\zeta_{X}(s - 1). 
$$
\end{cor}

\begin{rem}
\label{rem:sigmamotivic}
As in Remark~\ref{rem:totientmotivic}, we can view the formulas above as shadows of the formula $[\P^{1}] = [\A^{1}] + [\text{point}]$ in the appropriate Grothendieck ring. Further, formula (\ref{eq:divisorsumDirichletseries}) also suggests generalizing the divisor sum functions $\sigma_{m}(n)$ by $\sigma_{m,X}(z) = \#(\orb/z)^{m} + 1$ for each choice of object $X$, ring $\orb$ and monoid $Z$ with $z\in Z$. That is, a divisor sum function is a shadow of a class of the form $[\A^{m}] + [\text{point}]$ in some Grothendieck ring. 
\end{rem}

%-------------------------------------------------------------------------------------------------------------------------------------------------------------------------------------

\section{Psi Functions}
\label{sec:psi}

\subsection{Dedekind's $\psi$-function}

Dedekind originally introduced the following function to study modular forms, but it is an interesting arithmetic function in its own right. 
\begin{defn}
{\bf Dedekind's $\psi$-function} is the arithmetic function $\psi : \N\rightarrow\C$ defined by 
$$
\psi(n) = n\prod_{p\mid n} \left (1 + \frac{1}{p}\right ). 
$$
\end{defn}

From the definition, we can see that $\psi$ is multiplicative and for any prime $p$, $\psi(p) = p + 1$. In fact, these imply that for any squarefree integer $n$, $\psi(n) = \sigma_{1}(n)$. As arithmetic functions, $\psi = \id*|\mu|$, leading to the identity 
\begin{equation}
\sum_{n = 1}^{\infty} \frac{\psi(n)}{n^{s}} = \frac{\zeta_{\Q}(s)\zeta_{\Q}(s - 1)}{\zeta_{\Q}(2s)}. 
\end{equation}

On the other hand, $\psi(n)$ is the index of the congruence subgroup $\Gamma_{0}(n)$ in $SL_{2}(\Z)$, hence the connection to modular forms (or more directly, lattices). This interpretation is ripe for generalization, but let us first introduce a geometric interpretation of the $\psi$-function before moving to number fields, varieties, etc.

While the arithmetic functions $\phi$ and $\sigma_{1}$ had simple geometric interpretations, such an interpretation of $\psi$ is less obvious. After failing to find anything in the literature, we discovered the following interpretation, aided by a suggestion of David Zureick-Brown. 

First, notice that for a prime $p$, $\psi(p) = p + 1$ is the cardinality of the cyclic group $\F_{p^{2}}^{\times}/\F_{p}^{\times}$, where $\F_{q}$ denotes the field with $q$ elements. This generalizes as follows. For an integer $k\geq 1$ and a prime power $q = p^{a}$, let $\W_{k}(\F_{q})$ be the ring of $p$-typical Witt vectors of $\F_{q}$. Define the quotient groups 
$$
G(p^{k}) = \W_{k}(\F_{p^{2}})^{\times}/\W_{k}(\F_{p})^{\times}. 
$$

\begin{lem}
$\#G(p^{k}) = \psi(p^{k})$. 
\end{lem}

\begin{proof}
This follows from standard formulas for $\#\W_{k}(\F_{q})$ and $\#\W_{k}(\F_{q})^{\times}$. For more details, see \cite{kzb}.     
\end{proof}

For an arbitrary integer $n\geq 1$, with prime factorization $n = p_{1}^{k_{1}}\cdots p_{r}^{k_{r}}$, let 
$$
G(n) = \prod_{i = 1}^{r} G(p_{i}^{k_{i}}) = \prod_{i = 1}^{r} \W_{k_{i}}(\F_{p_{i}^{2}})^{\times}/\W_{k_{i}}(\F_{p_{i}})^{\times}. 
$$

\begin{cor}
$\#G(n) = \psi(n)$. 
\end{cor}

\begin{question}
\label{q:bigwitt}
Is there a group scheme $\mathcal{G}/\Spec\Z$ such that $G(n) \cong \mathcal{G}(\Z/n\Z)$ for all $n$? For example, is there an interpretation of the groups $G(n)$ in terms of the ring of big Witt vectors? 
\end{question}

We have not been able to devise such a group scheme $\mathcal{G}$, so this question is left for the curious reader. Ideally, $\mathcal{G}$ would also provide a geometric explanation for the identity $\psi = \operatorname{id}*|\mu|$. 

%-------------------------------------------------------------------------------------------------------------------------------------------------------------------------------------

\subsection{Number field $\psi$-function}

\begin{defn}
The {\bf $\psi$-function of a number field} $K/\Q$ is the function $\psi_{K} : I_{K}^{\times}\rightarrow\Z$ defined by 
$$
\psi_{K}(\frak{a}) = N(\frak{a})\prod_{\frak{p}\mid\frak{a}} \left (1 + \frac{1}{N(\frak{p})}\right ). 
$$
\end{defn}

We summarize the properties of $\psi_{K}$ here. 
\begin{prop}
\label{prop:numberfieldpsiprops}
Let $\frak{a},\frak{b},\frak{p}\subset\orb_{K}$ be ideals, with $\frak{p}$ prime. Then 
\begin{enumerate}[\quad (a)]
    \item If $\frak{a}$ and $\frak{b}$ are relatively prime, then $\psi_{K}(\frak{a}\frak{b}) = \psi_{K}(\frak{a})\psi_{K}(\frak{b})$. 
    \item $\psi_{K}(\frak{p}) = \sigma_{1,K}(\frak{p}) = N(\frak{p}) + 1$. In particular, if $\frak{a}$ is squarefree then $\psi_{K}(\frak{a}) = \sigma_{1,K}(\frak{a})$. 
    \item As arithmetic functions on $I_{K}^{\times}$, $\psi_{K} = N*|\mu_{K}|$, where $\mu_{K}$ is the M\"{o}bius function in $I(I_{K}^{\times})$. 
    \item $\psi(\frak{a}) = [SL_{2}(\orb_{K}) : \Gamma_{0}(\frak{a})]$, where $\Gamma_{0}(\frak{a})$ is the subgroup of matrices in $SL_{2}(\orb_{K})$ which are upper triangular mod $\frak{a}$. 
    \item $\psi_{K}(\frak{a}) = \#G(\frak{a})$ where 
    $$
    G(\frak{a}) := \prod_{\frak{p}^{k}\mid\mid\frak{a}} G(\frak{p}^{k}) := \prod_{\frak{p}^{k}\mid\mid\frak{a}} \W_{k}(\F_{N(\frak{p})^{2}})^{\times}/\W_{k}(\F_{N(\frak{p})})^{\times}. 
    $$
\end{enumerate}
\end{prop}

\begin{proof}
(a) -- (c) follow from the definition. (d) is proven in the usual way, i.e.~by counting the number of cyclic submodules of rank $N(\frak{a})$ in $E[\frak{a}]$, where $E$ is an elliptic curve over $K$ (equivalently, cyclic submodules in $E_{tors}$ which are isomorphic as $\orb_{K}$-modules to $\orb_{K}/\frak{a}$). Finally, (e) follows from the same standard formulas as above. 
\end{proof}

\begin{cor}
The Dirichlet series for $\psi_{K}$ satisfies the formula 
$$
\sum_{\frak{a}\subset\orb_{K}} \frac{\psi_{K}(\frak{a})}{N(\frak{a})^{s}} = \frac{\zeta_{K}(s)\zeta_{K}(s - 1)}{\zeta_{K}(2s)}
$$
where $\zeta_{K}(s)$ is the Dedekind zeta function of $K$. 
\end{cor}

\begin{proof}
This follows from (c) and the identity 
$$
\sum_{\frak{a}\subset\orb_{K}} \frac{|\mu_{K}(\frak{a})|}{N(\frak{a})^{s}} = \frac{\zeta_{K}(s)}{\zeta_{K}(2s)}
$$
which is a straightforward generalization of the same formula over $\Q$. 
\end{proof}

%-------------------------------------------------------------------------------------------------------------------------------------------------------------------------------------

\subsection{Geometric $\psi$-functions}

\begin{defn}
The {\bf $\psi$-function of a variety} $X/\F_{q}$ is the function $\psi_{X} : Z_{0}^{\eff}(X)\rightarrow\Z$ defined by 
$$
\psi_{X}(\alpha) = q^{\deg(\alpha)}\prod_{i = 1}^{r} (1 + q^{-\deg(x_{i})}). 
$$
\end{defn}

Once again, we have the following properties of $\psi_{X}$. 
\begin{prop}
Let $\alpha,\beta\in Z_{0}^{\eff}(X)$ and let $x\in|X|$. Then 
\begin{enumerate}[\quad (a)]
    \item If $\alpha$ and $\beta$ are relatively prime, then $\psi_{X}(\alpha + \beta) = \psi_{X}(\alpha)\psi_{X}(\beta)$. 
    \item $\psi_{X}(x) = \sigma_{1,X}(x) = q^{\deg(x)} + 1$. In particular, if $\alpha$ is a {\it log} effective $0$-cycle\footnote{A log cycle is one for which each irreducible cycle appears with coefficient $1$.} then $\psi_{X}(\alpha) = \sigma_{1,X}(\alpha)$. 
    \item As arithmetic functions on $Z_{0}^{\eff}(X)$, $\psi_{X} = \pi*|\mu_{X}|$, where $\mu_{X}$ is the M\"{o}bius function in $I(Z_{0}^{\eff}(X))$. 
    \item $\psi(\alpha) = [SL_{2}(\orb_{X,\alpha}) : \Gamma_{0}(I_{\alpha})]$, where $\Gamma_{0}(I_{\alpha})$ is the subgroup of matrices in $SL_{2}(\orb_{X,\alpha})$ which are upper triangular mod $I_{\alpha}$. 
    \item $\psi_{X}(\alpha) = \#G(\alpha)$ where 
    $$
    G(\alpha) := \prod_{kx\in\alpha} G(kx) := \prod_{kx\in\alpha} \W_{k}(\F_{q^{2\deg(\alpha)}})^{\times}/\W_{k}(\F_{q^{\deg(\alpha)}})^{\times}. 
    $$
\end{enumerate}
\end{prop}

\begin{proof}
(a) -- (c) and (e) are the same as in the proof of Proposition~\ref{prop:numberfieldpsiprops}. For (e), counting sublattices in $(\orb_{X,\alpha}/I_{\alpha})^{2}$ isomorphic to $\orb_{X,\alpha}/I_{\alpha}$ gives the result. 
\end{proof}

\begin{cor}
The power series for $\psi_{X}$ satisfies 
$$
\sum_{\alpha\in Z_{0}^{\eff}(X)} \psi_{X}(\alpha)t^{\deg(\alpha)} = \frac{Z(X,t)Z(X,qt)}{Z(X,t^{2})} = \frac{Z(X,t)Z(X\times\A^{1},t)}{Z(X_{\F_{q^{2}}},t)}
$$
where $X_{\F_{q^{2}}} = X\times_{\Spec\F_{q}}\Spec\F_{q^{2}}$. 
\end{cor}

We can extend this to the global setting by putting 
$$
\Psi_{X}(s) = \prod_{p\text{ prime}} \Psi(X_{p},p^{-s})
$$
for any $X/\Spec\Z$, where for each prime $p$, 
$$
\Psi(X_{p},t) = \sum_{\alpha\in Z_{0}^{\eff}(X_{p})} \psi_{X_{p}}(\alpha)t^{\deg(\alpha)}. 
$$

\begin{cor}
\label{cor:globalpsi}
For any arithmetic scheme $X$, we have 
$$
\Psi_{X}(s) = \frac{\zeta_{X}(s)\zeta_{X\times\A^{1}}(s)}{\zeta_{X}(2s)}. 
$$
\end{cor}

As suggested by Question~\ref{q:bigwitt}, it would be desirable to have a motivic explanation for Corollary~\ref{cor:globalpsi}, along the lines of Remarks~\ref{rem:totientmotivic} and~\ref{rem:sigmamotivic}.

%-------------------------------------------------------------------------------------------------------------------------------------------------------------------------------------

\section{Other Arithmetic Functions}
\label{sec:other}

\subsection{Liouville's $\lambda$-function}
\label{sec:lambda}

\begin{defn}
{\bf Liouville's $\lambda$-function} is the arithmetic function $\lambda : \N\rightarrow\C$ defined by $\lambda(1) = 1$ and 
$$
\lambda(n) = \begin{cases}
  1, &\text{if $n$ is a product of an even number of primes}\\
  -1, &\text{if $n$ is a product of an odd number of primes}. 
\end{cases}
$$
That is, $\lambda(n) = (-1)^{\Omega(n)}$ where $\Omega\left (\prod p_{i}^{a_{i}}\right ) = \sum a_{i}$. 
\end{defn}

It is well-known that $\lambda$ is completely multiplicative: $\lambda(nm) = \lambda(n)\lambda(m)$ for \emph{any} $n,m\in\N$. We also have that $\lambda(n) = \mu(n')$ where $n'$ is the squarefree part of $n$, i.e.~the unique squarefree positive integer $n'$ such that $n = a^{2}n'$ for some $a\in\N$. As a consequence, 
\begin{equation}
\sum_{d\mid n} \lambda(d) = \begin{cases}
  1, &\text{if $n$ is a square}\\
  0, &\text{otherwise}. 
\end{cases}
\end{equation}
Applying M\"{o}bius inversion yields 
\begin{equation}
\lambda(n) = \sum_{d^{2}\mid n} \mu(n/d^{2}). 
\end{equation}

Another feature is that $\lambda*|\mu| = \delta$, i.e.~$\lambda$ is the convolution inverse of $|\mu|$, just as $\zeta$ is the convolution inverse of $\mu$. On the level of Dirichlet series, we have 
\begin{equation}
\sum_{n = 1}^{\infty} \frac{\lambda(n)}{n^{s}} = \frac{\zeta_{\Q}(2s)}{\zeta_{\Q}(s)}. 
\end{equation}

\begin{problem}
\label{prob:lambda}
Extend $\lambda$ to a number field $K/\Q$ by setting 
$$
\lambda_{K}(\frak{a}) = (-1)^{\Omega(\frak{a})}
$$
where $\Omega\left (\prod \frak{p}_{i}^{a_{i}}\right ) = \sum a_{i}$. Show $\lambda_{K}*|\mu_{K}| = \delta$ in the incidence algebra $I(I_{K}^{\times})$, which implies 
$$
\sum_{\frak{a}\subset\orb_{K}} \frac{\lambda_{K}(\frak{a})}{N(\frak{a})^{s}} = \frac{\zeta_{K}(2s)}{\zeta_{K}(s)}. 
$$
Repeat for a variety $X/\F_{q}$ and for a scheme $X/\Spec\Z$. 
\end{problem}

\begin{problem}
If possible, find a class $\xi$ in the Grothendieck ring of schemes such that $\#\xi(\Z/n\Z) = \lambda(n)$ and similarly for each generalization of $\psi$ in Problem~\ref{prob:lambda}. 
\end{problem}

%-------------------------------------------------------------------------------------------------------------------------------------------------------------------------------------

\subsection{Sums of squares}
\label{sec:sumofsquares}

A famous number theory identity says that 
\begin{equation}
r_{2}(n) = 4\sum_{d\mid n} \chi_{-1}(d)
\end{equation}
where $r_{2}(n)$ counts the number of ways $n$ can be written as a sum of two squares and $\chi_{-1} = \legen{-1}{\cdot}$ is unique nontrivial quadratic Dirichlet character mod $4$. That is, $r_{2} = 4\zeta*\chi_{-1}$. 

\begin{question}
Does this identity generalize to (i) number fields? (ii) higher dimensional geometric objects? (iii) a formula in some Grothendieck ring? What about formulas involving the sum of squares functions $r_{k}(n)$ for $k\geq 3$? 
\end{question}

%-------------------------------------------------------------------------------------------------------------------------------------------------------------------------------------

%-------------------------------------------------------------------------------------------------------------------------------------------------------------------------------------

\end{document}